\documentclass[11pt]{article}
\usepackage[T1]{fontenc}
\usepackage{lmodern}

\usepackage[letterpaper,margin=1in]{geometry}
\usepackage{amsmath,amssymb,mathtools,amsthm,bm,mathrsfs}
\IfFormatAtLeastTF{2026-06-01}{}{\usepackage{aliascnt}}
\usepackage{booktabs,tabularx,array}
\usepackage{enumitem}
\usepackage[nopatch=footnote]{microtype}
\usepackage{xcolor}
\usepackage[numbers,sort&compress]{natbib}
\usepackage{hyperref}
\usepackage[nameinlink,noabbrev]{cleveref}
\usepackage{setspace}

\allowdisplaybreaks[3]
\numberwithin{equation}{section}
\setlist[itemize]{leftmargin=1.65em,itemsep=0.16em,topsep=0.25em}
\setlist[enumerate]{leftmargin=1.9em,itemsep=0.18em,topsep=0.25em}
\hypersetup{hidelinks,pdftitle={Joint Zeroth-Order Lower Bounds for Goldstein Stationarity on Euclidean Balls}}

\newtheorem{theorem}{Theorem}[section]
\IfFormatAtLeastTF{2026-06-01}{%
  \newcommand{\sharedtheorem}[2]{\newtheorem{#1}[theorem]{#2}}%
}{%
  \newcommand{\sharedtheorem}[2]{%
    \newaliascnt{#1}{theorem}%
    \newtheorem{#1}[#1]{#2}%
    \aliascntresetthe{#1}}%
}
\sharedtheorem{lemma}{Lemma}
\sharedtheorem{proposition}{Proposition}
\sharedtheorem{corollary}{Corollary}
\sharedtheorem{definition}{Definition}
\sharedtheorem{assumption}{Assumption}
\theoremstyle{remark}
\sharedtheorem{remark}{Remark}

\newcommand{\R}{\mathbb{R}}

\newcommand{\E}{\mathbb{E}}
\newcommand{\Prob}{\mathbb{P}}

\newcommand{\clconv}{\operatorname{cl\,conv}}

\newcommand{\dist}{\operatorname{dist}}
\newcommand{\Lip}{\operatorname{Lip}}
\newcommand{\rank}{\operatorname{rank}}
\newcommand{\tr}{\operatorname{tr}}

\newcommand{\proj}{\operatorname{proj}}
\newcommand{\Normal}{\mathcal{N}}
\newcommand{\eps}{\varepsilon}
\newcommand{\inner}[2]{\left\langle #1,#2\right\rangle}
\newcommand{\norm}[1]{\left\lVert #1\right\rVert}
\newcommand{\op}{\mathrm{op}}

\title{\textbf{Joint Lower Bounds for Zeroth-Order Nonconvex Optimization on Euclidean Balls}}
\newcommand{\authornames}{Haihan Zhang, Wendao Wu, Chenheng Zhang, Yanyi Li, Chunyuan Zheng, Cong Fang, Haoxuan Li, Zhouchen Lin}
\newcommand{\affiliation}{Peking University}
\newcommand{\contactemail}{colombo.2026@outlook.com}
\newcommand{\mail}[1]{\href{mailto:#1}{\texttt{#1}}}
\newif\ifanonymous
\anonymousfalse
\ifanonymous
\author{Anonymous Authors}
\hypersetup{pdfauthor={Anonymous Authors}}
\else
\author{Haihan Zhang$^{1,*}$\quad Wendao Wu$^{1,*}$\quad Chenheng Zhang$^{1,*}$\\[0.20em]
Yanyi Li$^1$\quad Chunyuan Zheng$^1$\\[0.20em]
Cong Fang$^{1}$\quad Haoxuan Li$^{1,\dagger}$\quad Zhouchen Lin$^{1,\dagger}$\\[0.60em]
\small $^1$\affiliation\\[0.35em]
\footnotesize\mail{\contactemail}\\[0.35em]
\small $^*$Equal contribution.\quad $^\dagger$Corresponding authors.}
\hypersetup{pdfauthor={\authornames}}
\fi
\date{}
\newcommand{\ResearchAgentSystem}{our laboratory's internal auto-research system}

\begin{document}
\maketitle

\begin{abstract}
We prove a joint stochastic zeroth-order lower bound for Goldstein stationarity on a Euclidean query ball, even when the ball is guaranteed to contain a stationary point. In dimension \(d\), let \(f=\E F(\cdot;\xi)\), assume \(\E\Lip(F(\cdot;\xi))^2\le L_0^2\), and bound the initial objective gap over the ball by \(\Delta\). For neighborhood radius \(\delta>0\) and residual tolerance \(\eps>0\), our smooth hard family requires
\[
  \Omega\!\left(\frac{dL_0^2\Delta}{\delta\eps^3}\right)
\]
scalar evaluations for success probability \(1/2\), against randomized adaptive algorithms that may retain and repeatedly query each sampled function. The result holds for \(\eps\le cL_0\), \(\Delta\ge C\delta\eps\), and \(d\ge C[1+\log(2+\Delta L_0^2/(\delta\eps^3))]\), on a constructed ball of radius \(\Theta(\Delta/\eps)\). A sequence of localized regions forces repeated direction estimation, and an adaptive Gaussian posterior argument controls sample reuse. The result establishes a joint dimension and accuracy obstruction on solvable bounded-domain instances. The gap is local to the query ball; a matching minimax characterization at a common radius and the corresponding unrestricted global-gap lower bound remain open here.
\end{abstract}

\medskip
{\small
\begin{samepage}
\noindent\textbf{AI Usage.}
Nearly the entire research pipeline for this paper was carried out by
\ResearchAgentSystem{}, powered by GPT-5.6 Sol. The system also conducted a
Lean-backed article audit of the resulting manuscript. The authors subsequently reviewed and approved the
mathematical claims, presentation, and formal artifacts, and take
responsibility for the final manuscript. The complete Lean audit report and the system's technical report will be made public at a later date.
\par
\end{samepage}
}

\paragraph{Publication Intent and Recognition of Independent Work.}

This preprint is released as a research output of our study of
automated scientific research, for public scrutiny and discussion.
As we do not intend to submit this manuscript to a journal or
conference, it has not undergone the full editorial refinement
and internal review we would normally require for submission.
The authors remain responsible for its contents and for correcting
any errors or inaccurate attribution.

We recognize that releasing these results may affect researchers
who have independently been pursuing the same questions, including
work that is not yet public.
If this applies to your work, we invite you to contact us at
\texttt{colombo.2026@outlook.com}.
A brief description of the overlap, together with any dated
materials you are comfortable sharing, would help us understand
its scope and chronology.
Such materials might include draft excerpts, presentation slides,
or revision records; a complete manuscript or unpublished proof
is not required to begin the conversation.

We will review these cases carefully and update citations,
acknowledgments, and accounts of independent or concurrent work
as appropriate, including revising priority claims where warranted.
Nonpublic materials will be treated confidentially and used only
to assess the overlap unless their authors authorize otherwise.
We will obtain explicit permission from the authors before
publicly discussing such materials or providing them to an
automated research system.

\section{Introduction}\label{sec:introduction}

We study stochastic zeroth-order optimization of a Lipschitz, possibly nonsmooth and nonconvex objective \(f(x)=\E_\xi F(x;\xi)\). Queries and the output lie in a known Euclidean ball \(X=B_2^d(x_0,R)\). Given \(\delta,\eps>0\), the goal is to find
\[
  x_{\rm out}\in X
  \quad\text{such that}\quad
  G_\delta(f,x_{\rm out})
  :=\dist\!\left(0,\clconv\!\bigcup_{\|y-x_{\rm out}\|\le\delta}
  \partial^C f(y)\right)\le\eps,
\]
where \(\partial^C f\) is the Clarke subdifferential. This Goldstein criterion allows nearby subgradients to cancel~\cite{Goldstein1977,Clarke1990}. The algorithm observes scalar sample-function values. Our lower bound applies even with the promise that \(X\) contains an ambient stationary point; the boundary of \(X\) is not itself a stationarity certificate.

Function values provide only scalar information about a local geometry that lives in \(d\) dimensions. Nevertheless, evaluating the same sampled function at nearby points allows algorithms to estimate useful directions without observing derivatives. Lin, Zheng, and Jordan~\cite{LinEtAl2022} established such guarantees for Goldstein stationarity, and Chen, Xu, and Luo~\cite{ChenEtAl2023} improved their accuracy dependence using variance reduction. Kornowski and Shamir~\cite{KornowskiShamir2024} subsequently obtained the scalar-evaluation bound
\begin{equation}
 O\!\left(\frac{dL_0^2\Delta}{\delta\eps^3}\right),
 \label{eq:benchmark}
\end{equation}
where \(L_0\) bounds the root mean square of the sample Lipschitz constants and \(\Delta\) bounds the initial objective gap. Their guarantee assumes a global gap and permits unrestricted queries. The progression of these rates is summarized in \Cref{tab:comparison}.

The product structure of \eqref{eq:benchmark} raises a lower-bound question. A separate obstruction for dimension and a separate obstruction for accuracy need not occur on the same instance: together, such bounds can force only the larger of the two costs, rather than their product. Kornowski and Shamir~\cite{KornowskiShamir2024} explicitly identify this distinction in discussing the optimality of their rate. Establishing a joint lower bound requires an instance on which the algorithm repeatedly pays the dimension-dependent cost as it makes progress toward stationarity. It must also account for sample reuse, since an algorithm can retain a sampled function and interrogate it at many adaptively chosen points.

\paragraph{A joint lower bound on Euclidean balls.}
We establish a lower bound with the full product in \eqref{eq:benchmark} in a bounded query model, even on instances guaranteed to have a stationary point in the ball. Here the initial gap \(\Delta\) is measured over the query ball, and the hard instance has radius \(R=\Theta(\Delta/\eps)\). By \Cref{cor:goldstein-lower}, every randomized adaptive algorithm needs \(\Omega(dL_0^2\Delta/(\delta\eps^3))\) scalar evaluations to attain Goldstein stationarity with probability at least \(1/2\). The result holds when \(\eps\le cL_0\), \(\Delta\ge C\delta\eps\), and the dimension exceeds a logarithmic threshold, stated in the corollary and in normalized form in \Cref{tab:comparison}. The oracle permits arbitrary reuse of each independently sampled function, charging one unit for each scalar evaluation. Thus even adaptive reuse cannot avoid the joint dimension and accuracy cost on this family.

The existence promise is essential to the interpretation of this result. A local gap alone does not ensure an ambient stationary point in a small query ball: a linear objective can satisfy the gap bound while having a nonzero gradient everywhere. Our construction instead places a stationary point inside the ball and makes its location difficult to discover. The local gap and bounded query domain distinguish our theorem from the global-gap benchmark in \eqref{eq:benchmark}; we do not claim a matching minimax characterization or an unrestricted global-gap lower bound.

\begin{table}[t]
\centering
\small
\setlength{\belowcaptionskip}{8pt}
\caption{Scalar-evaluation complexity for stochastic Goldstein stationarity, with \(L_0=\Delta=1\), \(0<\delta,\eps<1\), and constant success probability in every row. Our lower bound additionally requires \(\eps\le c\), \(\delta\eps\le 1/C\), and \(d\ge C[1+\log(2+1/(\delta\eps^3))]\). Prior upper bounds use a global gap and unrestricted queries. Our lower bound uses a local gap and promises a stationary point in the constructed query ball. These different settings do not establish a matching minimax rate.}
\label{tab:comparison}
\renewcommand{\arraystretch}{1.3}
\begin{tabularx}{\textwidth}{@{}>{\raggedright\arraybackslash}p{0.24\textwidth}>{\centering\arraybackslash}p{0.22\textwidth}>{\centering\arraybackslash}p{0.18\textwidth}>{\raggedright\arraybackslash}X@{}}
\toprule
Work & Upper bound & Lower bound & Access and domain \\
\midrule
Lin et al.~\cite{LinEtAl2022} & \(O(d^{3/2}\delta^{-1}\eps^{-4})\) & --- & Fresh common-sample pairs; unrestricted queries \\
Chen et al.~\cite{ChenEtAl2023} & \(O(d^{3/2}\delta^{-1}\eps^{-3})\) & --- & Common-sample differences; unrestricted queries \\
Kornowski--Shamir~\cite{KornowskiShamir2024} & \(O(d\delta^{-1}\eps^{-3})\) & --- & Fresh common-sample pairs; unrestricted queries \\
\midrule
\textbf{This Paper}\newline(\Cref{cor:goldstein-lower}) & --- & \(\boldsymbol{\Omega(d\delta^{-1}\eps^{-3})}\) & Reusable sample handles; stationary point promised; constructed \(R=\Theta(\eps^{-1})\) \\
\bottomrule
\end{tabularx}
\end{table}

\paragraph{Forcing the dimension cost at each stage.}
The central difficulty is to prevent one scalar observation from revealing useful information about several stages at once. A derivative zero-chain controls which coordinates appear in a gradient, but this alone does not control how an exact function value combines hidden information. Our construction instead separates the stages spatially. Each of a sequence of disjoint regions contains information about the transverse displacement needed to reach the next region. This makes progress depend on learning a new direction locally, while keeping the entire objective within the prescribed Lipschitz and gap budgets.

Sample reuse creates a second difficulty: repeated evaluations of one sampled function are correlated, so they cannot be treated as fresh noisy observations. In our construction, each retained sample contains a fixed Gaussian perturbation of the direction hidden in a region. Queries to that sample reveal projections of the same noisy vector. An adaptive Gaussian posterior argument tracks these projections and quantifies the information needed to locate the next region, regardless of how evaluations are allocated across retained samples. A shadow experiment, in which unresolved future regions are suppressed, then controls the possibility of bypassing this learning process. Until the first disagreement, its transcript agrees with the actual one; the posterior bound makes a premature visit to an unresolved region unlikely. The argument covers both adaptive queries and an arbitrary final output.

For an objective with \(H\)-Lipschitz gradient and a target gradient norm \(\eta\), the construction has \(\Theta(H\Delta/\eta^2)\) stages, each requiring \(\Theta(dL_0^2/\eta^2)\) scalar observations in the information argument. Their product gives the smooth lower bound in \Cref{thm:smooth-lower}. Taking \(H=\eps/\delta\) and \(\eta=2\eps\) transfers this obstruction to Goldstein stationarity: smoothness ensures that a Goldstein residual at most \(\eps\) implies gradient norm at most \(\eps+H\delta\).

\section{Related Work}\label{sec:related}

\subsection{Goldstein stationarity}

Goldstein~\cite{Goldstein1977} introduced the enlarged subdifferential used
in this paper; we follow the Clarke formulation~\cite{Clarke1990}.  Zhang
et al.~\cite{ZhangEtAl2020} gave the first finite-time guarantees for
Goldstein-type stationarity in nonsmooth nonconvex optimization.  Later work
developed implementable gradient-sampling and perturbation methods and
removed restrictive oracle assumptions~\cite{TianEtAl2022,DavisEtAl2022}.
Cutkosky, Mehta, and Orabona~\cite{CutkoskyEtAl2023} obtained the optimal
stochastic first-order accuracy dependence through an online-to-nonconvex
reduction.  These first-order methods observe gradients or subgradients and therefore do
not address the scalar-information bottleneck studied here. More recently, Kornowski~\cite{Kornowski2026Stationarity} proved a deterministic-objective lower bound of \(\Omega(\Delta L_0^2/(\delta\eps^3))\) for zero-respecting first-order algorithms, in a dimension constructed as a function of the accuracy parameters. That theorem has a global gap and does not assert our additional factor of a freely specified dimension or coverage of arbitrary adaptive value-query algorithms.

\subsection{Stochastic zeroth-order algorithms}

Lin et al.~\cite{LinEtAl2022} combine uniform smoothing with a common-sample
two-point estimator and stochastic gradient updates on the smoothed objective.
Chen et al.~\cite{ChenEtAl2023} reduce the estimation variance through recursive
differences of zeroth-order gradient estimates at consecutive iterates. Their
method reuses a sample across those estimates, illustrating why the number of
scalar evaluations must be distinguished from the number of independent samples.
Kornowski and Shamir~\cite{KornowskiShamir2024} instead use a Goldstein-subdifferential
inclusion to apply an online-to-nonconvex method directly to a smoothed objective.
This avoids the extra dimension dependence associated with optimizing its
gradient norm. Their estimator uses a fresh sample at two symmetric points.
The comparison in \Cref{tab:comparison} records the resulting rates; our reusable-handle
model permits all of these sample-access patterns. Liu and Luo~\cite{LiuLuo2026HeavyTail} extend this line to sample Lipschitz constants with bounded \(p\)-th moment, \(1<p\le2\), under an additional regularity assumption and a global lower bound. Their \(p=2\) specialization has the same principal dimension and accuracy exponents. These algorithmic results do not supply a lower bound for arbitrary sample reuse.

Two recent extensions change the feedback model or the stationarity target. El Bakkali, Chayti, and Saadi~\cite{ElBakkaliEtAl2026Comparison} study Goldstein stationarity using noisy pairwise comparisons governed by a known link function, rather than scalar sample-function values. Lei et al.~\cite{LeiEtAl2026SecondOrder} introduce a second-order Goldstein criterion for continuously differentiable functions with locally Lipschitz gradients and analyze a Gaussian-smoothing method under a coercivity-type assumption. These upper-bound results concern different oracle or regularity conditions and do not establish the joint lower bound considered here.

\subsection{Lower-bound context}

Classical zeroth-order lower bounds show that scalar feedback incurs a
dimension penalty even for convex objectives~\cite{DuchiEtAl2015}.  Smooth
stochastic nonconvex lower bounds identify the \(\eps^{-4}\) barrier for
first-order noisy optimization~\cite{ArjevaniEtAl2023}, but those bounds do not by themselves impose the additional dimension cost for scalar sample values.  Recent exact-value lower bounds
for convex optimization on Euclidean balls use a different target
(suboptimality) and a fixed deterministic objective rather than stochastic
sample functions~\cite{ZhangEtAl2026ExactValue}. Our nonconvex family couples
the scalar-information cost with the stochastic Lipschitz scale, smoothness,
initial gap, and Goldstein parameters on a family with a stationary point in the query ball. A separate recent result of Kornowski~\cite{Kornowski2026Determinism} proves an exponential dimension obstruction for deterministic algorithms. Its algorithm restriction distinguishes it from the randomized reusable-handle model considered here. These comparisons do not identify our bounded-domain lower bound with an unrestricted global-gap result.

\section{Model and Complexity Measures}\label{sec:model}

\subsection{Bounded query domain and stochastic objectives}

Fix a starting point \(x_0\in\R^d\) and a Euclidean query ball
\[
  X=B_2^d(x_0,R).
\]
An instance consists of a probability space and a jointly measurable sample
function \(F:\R^d\times\Xi\to\R\).  We assume
\(\E|F(x_0;\xi)|<\infty\).  There is a measurable random variable
\(L(\xi)\ge0\) such that, on a common full-measure event,
\begin{equation}
  |F(x;\xi)-F(y;\xi)|\le L(\xi)\norm{x-y}_2
  \quad\text{for all }x,y\in\R^d,
  \qquad
  \E L(\xi)^2\le L_0^2.
  \label{eq:source-3-1}
\end{equation}
The population objective is
\[
  f(x)=\E F(x;\xi),
\]
and the bounded-domain gap condition is
\begin{equation}
  f(x_0)-\min_{x\in X}f(x)\le\Delta.
  \label{eq:source-3-2}
\end{equation}
Queries and the gap condition are restricted to \(X\), while the stationarity target uses the ambient Goldstein neighborhood. For an output near the boundary, this neighborhood may extend outside \(X\), so the objective there can affect the residual. Defining the sample functions on \(\R^d\) specifies this target unambiguously.

For the smooth lower-bound problem we use the globally smooth subclass:
\(f\) is continuously differentiable and
\begin{equation}
  \norm{\nabla f(x)-\nabla f(y)}_2\le H\norm{x-y}_2
  \quad\text{for all }x,y\in\R^d.
  \label{eq:source-3-3}
\end{equation}
This stronger formulation avoids any boundary ambiguity in the later
Goldstein reduction.  The hard family constructed below is \(C^\infty\) and
satisfies \eqref{eq:source-3-3} globally.

\subsection{Reusable sample-function value oracle}

The algorithm may maintain any finite collection of sample handles.  On
request, a new handle \(h\) is assigned an independent draw
\(\xi_h\sim P\).  At each evaluation the algorithm selects a retained handle
\(h\) and a point \(x\in X\), both adaptively from the entire transcript,
and observes the exact scalar
\[
  F(x;\xi_h).
\]
Each scalar observation costs one unit.  Drawing a handle without evaluating
it is free and conveys no information.  The output must lie in \(X\).

This oracle includes as a special case the usual fresh common-sample pair:
draw a new handle, evaluate \(F(c+v;\xi_h)\) and
\(F(c-v;\xi_h)\), and discard the handle.  A lower bound for the reusable
oracle is therefore at least as strong as a lower bound for fresh pairs.

\subsection{Stationarity targets}

For a locally Lipschitz objective, let \(\partial^C f(x)\) denote the Clarke
subdifferential.  The closed-ball Goldstein subdifferential and residual are
\begin{equation}
  \partial_\delta f(x)
  =\clconv\!\bigcup_{\norm{y-x}_2\le\delta}\partial^C f(y),
  \qquad
  G_\delta(f,x)=\dist_2(0,\partial_\delta f(x)).
  \label{eq:source-3-4}
\end{equation}
For a smooth objective we also consider the ordinary target
\(\norm{\nabla f(x)}_2\le\eta\).

Let \(N_{\nabla}(d,R,L_0,H,\Delta,\eta)\) be the least integer \(N\) for
which one randomized algorithm, using at most \(N\) scalar evaluations on
every transcript, outputs \(x_{\rm out}\in X\) satisfying
\begin{equation}
  \Prob\!\left(\norm{\nabla f(x_{\rm out})}_2\le\eta\right)\ge\frac12
  \label{eq:source-3-5}
\end{equation}
uniformly over the smooth instance class.  Define
\(N_G(d,R,L_0,\Delta,\delta,\eps)\) analogously with the event
\(G_\delta(f,x_{\rm out})\le\eps\). If no such algorithm exists, the corresponding complexity is \(+\infty\).

For the lower bounds, we also consider the subclass with the promise
\[
  \text{there exists }x_*\in X\text{ such that }0\in\partial^C f(x_*).
\]
For smooth objectives this means \(\nabla f(x_*)=0\). Write
\(N_{\nabla}^{\mathrm{stat}}\) and \(N_G^{\mathrm{stat}}\) for the same complexity definitions restricted to these subclasses. The algorithm is informed of the promise, but is not given \(x_*\). This prevents the lower bound from being explained merely by absence of a solution. Since the subclass is smaller, \(N_G^{\mathrm{stat}}\le N_G\), and similarly in the smooth case. The auxiliary upper bound in \Cref{sec:upper} applies to the full class at its stated radius.

\section{Main Results}\label{sec:results}

Our first theorem gives the smooth core of the lower bound.

\begin{theorem}[Product lower bound on a Euclidean ball]
\label{thm:smooth-lower}
Let \(d\ge2\) and \(L_0,H,\Delta,\eta>0\).  There are absolute constants
\(c_0,C_0,c_N,c_R,C_R>0\) such that the following holds.  Suppose
\begin{equation}
  0<\eta\le c_0L_0,
  \qquad
  H\Delta\ge C_0\eta^2,
  \qquad
  d\ge C_0\!\left[1+\log\!\left(2+\frac{H\Delta L_0^2}{\eta^4}\right)\right].
  \label{eq:source-4-1}
\end{equation}
Then there is a deterministic radius \(R\) satisfying
\begin{equation}
  c_R\frac{\Delta}{\eta}
  \le R\le
  C_R\frac{\Delta}{\eta}
  \label{eq:source-4-2}
\end{equation}
and a family of \(C^\infty\) stochastic objectives on
\(X=B_2^d(0,R)\) such that, for every member of the family,
\begin{equation}
  \Lip(\nabla f;\R^d)\le H,
  \qquad
  f(0)-\min_X f\le\Delta,
  \qquad
  \E\Lip(F(\cdot;\xi))^2\le L_0^2.
  \label{eq:source-4-3}
\end{equation}
Every member also has a point \(x_*\in X\) with \(\nabla f(x_*)=0\). Even when this existence is promised, for every adaptive randomized reusable-handle algorithm using at most
\begin{equation}
  N\le c_N\frac{dL_0^2H\Delta}{\eta^4}
  \label{eq:source-4-4}
\end{equation}
scalar evaluations, there is a fixed instance in this family for which
\[
  \Prob\!\left(\norm{\nabla f(x_{\rm out})}_2\le\eta\right)<\frac12.
\]
In particular,
\(N_{\nabla}^{\mathrm{stat}}(d,R,L_0,H,\Delta,\eta)
 =\Omega(dL_0^2H\Delta/\eta^4)\).
\end{theorem}

The theorem directly implies the desired Goldstein lower bound.

\begin{corollary}[Goldstein lower bound on a Euclidean ball]
\label{cor:goldstein-lower}
Let \(d\ge2\) and \(L_0,\Delta,\delta,\eps>0\).  There are absolute
constants \(c,C>0\) such that, whenever
\begin{equation}
  0<\eps\le cL_0,
  \qquad
  \Delta\ge C\delta\eps,
  \qquad
  d\ge C\!\left[1+\log\!\left(2+
  \frac{\Delta L_0^2}{\delta\eps^3}\right)\right],
  \label{eq:source-4-5}
\end{equation}
there is a radius \(R=\Theta(\Delta/\eps)\) and a valid bounded-domain
instance family, every member of which contains a stationary point in the ball, for which
\begin{equation}
  N_G^{\mathrm{stat}}(d,R,L_0,\Delta,\delta,\eps)
  \ge
  c\frac{dL_0^2\Delta}{\delta\eps^3}.
  \label{eq:source-4-6}
\end{equation}
The lower bound holds even if sample handles may be retained and queried
arbitrarily many times.
\end{corollary}

The lower bound is a statement about difficulty of locating an existing stationary point. For the full local-gap class, small radii can instead make the target impossible. The following elementary observation distinguishes these two sources of difficulty.

\begin{proposition}[Radius obstruction without an existence promise]
\label{prop:radius-obstruction}
If \(L_0>\eps\) and \(0<R<\Delta/\eps\), then
\[
  N_G(d,R,L_0,\Delta,\delta,\eps)=+\infty.
\]
The analogous conclusion holds for \(N_{\nabla}\) with \(\eta\) in place of \(\eps\), for any \(H>0\).
\end{proposition}
\begin{proof}
Choose \(a\) with \(\eps<a\le\min\{L_0,\Delta/R\}\), and take the deterministic sample function \(F(x;\xi)=f(x)=-a\langle e_1,x-x_0\rangle\). Its local gap is \(aR\le\Delta\), its Lipschitz constant is \(a\le L_0\), and \(G_\delta(f,x)=a>\eps\) at every point. It is smooth with constant gradient, which also proves the smooth assertion. No algorithm can output a successful point on this instance.
\end{proof}

Our construction uses \(\gamma=4\eta\) and a radius large enough to contain every local bump. At the point \(Re_0\), beyond their clock-coordinate supports, its value is \(-\gamma R\), whereas its initial value is zero. Thus its gap condition forces \(R\le\Delta/(4\eta)\), and the Goldstein specialization has \(R\le\Delta/(8\eps)\). The existence promise in \Cref{thm:smooth-lower,cor:goldstein-lower} therefore matters: it excludes the no-solution instances responsible for the infinite full-class complexity. In particular, the rate
\begin{equation}
  \frac{dL_0^2\Delta}{\delta\eps^3},
  \label{eq:source-4-10}
\end{equation}
is a lower bound on the promised class, not a tight minimax characterization at a common radius.

\paragraph{Remark (localizing the known upper bound).}
For completeness, \Cref{sec:upper} shows that the Kornowski--Shamir method~\citep{KornowskiShamir2024} retains its evaluation rate using only a local gap when \(R\ge C(\Delta/\eps+\delta)\). This supplies a bounded-query guarantee at larger radii, without improving the established complexity order or matching the radius of our hard instances.

\section{Lower-Bound Construction}\label{sec:lower-construction}

We first explain how the construction forces repeated local learning, then define the hard family and verify its admissibility. The information argument in \Cref{sec:lower-proof} will turn this geometry into a query lower bound.

\subsection{Proof overview}\label{sec:architecture}

The product in \Cref{thm:smooth-lower} comes from multiplying a sequential depth by a per-stage scalar-information cost. The construction and admissibility argument are formalized in \Cref{lem:class-verification}; \Cref{lem:posterior,lem:small-ball,lem:early-hit} supply the information and adaptive traversal estimates.

The hard family has three components: a linear carrier that prevents stationarity, disjoint informative bumps that hide the next region, and a terminal bump that creates a stationary point. The bumps must be separated because scalar values can aggregate information from several hidden directions. For example, if \(S=\sum_{j=1}^M U_j\) with independent \(U_j\sim\Normal(0,I_d/d)\), a single observation \(Y=\langle S,v\rangle\) along a fixed unit vector satisfies
\begin{equation}
  \E[S\mid Y]=Yv,\qquad \E\|\E[S\mid Y]\|_2^2=M/d.
  \label{eq:source-5-1}
\end{equation}
In our construction, a value in region \(j\) instead depends only on its own hidden direction \(U_j\). Learning this direction locates the next region. Reusing a sample reveals projections of one fixed noisy vector, so the posterior calculation in \Cref{lem:posterior,lem:small-ball} determines how many scalar observations are needed for this localization.

Two costs then multiply: the gap budget permits \(M=\Theta(H\Delta/\eta^2)\) stages, and resolving each stage costs \(q_0=\Theta(dL_0^2/\eta^2)\) observations. The essential remaining step is to rule out skipping a stage. \Cref{lem:early-hit} does so through a shadow transcript that suppresses unresolved bumps; a first-disagreement coupling transfers this estimate to the actual oracle. The same argument includes the final output, which need not be a query point.

\subsection{The hidden-ball Gaussian chain}\label{sec:construction}

Write \(d=n+1\) and decompose \(x=(s,z)\in\R\times\R^n\), where \(s\) is
the clock coordinate.  Choose a fixed smooth function \(\psi:[0,\infty)\to[0,1]\) equal to one on \([0,1/16]\) and zero on \([1,\infty)\), and set \(\chi(y)=\psi(\|y\|_2^2)\). This gives a smooth cutoff \(\chi\in C_c^\infty(\R^{n+1})\) such that
\begin{equation}
  0\le\chi\le1,
  \qquad
  \chi(y)=1\ \text{for }\norm{y}_2\le\frac14,
  \qquad
  \chi(y)=0\ \text{for }\norm{y}_2\ge1.
  \label{eq:source-6-1}
\end{equation}
The informative bump encodes a transverse direction; the terminal bump cancels the carrier near its center. For \(v\in\R^n\), define
\begin{equation}
  b_v(y)=\chi(y)\inner{v}{y_\perp},
  \qquad
  \tau_0(y)=\chi(y)y_0.
  \label{eq:source-6-2}
\end{equation}
The first two derivatives of this fixed radial profile give operator-norm bounds independent of the dimension. In particular, there is an absolute constant \(C_\chi\ge1\) such that
\begin{equation}
\begin{aligned}
  |b_v(y)|&\le \norm{v}_2,
  &\qquad |\tau_0(y)|&\le1,\\
  \norm{\nabla b_v(y)}_2&\le C_\chi\norm{v}_2,
  &
  \norm{\nabla^2b_v(y)}_{\op}&\le C_\chi\norm{v}_2,\\
  \norm{\nabla\tau_0(y)}_2&\le C_\chi,
  &
  \norm{\nabla^2\tau_0(y)}_{\op}&\le C_\chi
\end{aligned}
  \label{eq:source-6-3}
\end{equation}
for all \(y\).

Choose the carrier slope \(\gamma\) above the target \(\eta\), the room radius \(\rho\) large enough for \(H\)-smoothness, and the informative amplitude \(\lambda\) small enough to preserve descent outside the terminal room. The clock spacing makes supports disjoint, while the transverse step makes a future hit require an accurate direction estimate. Set
\begin{equation}
  \gamma=4\eta,
  \qquad
  \rho=C_\rho\frac{\gamma}{H},
  \qquad
  \ell_{\rm clk}=4\rho,
  \qquad
  a_{\rm step}=16\rho,
  \qquad
  \lambda=\frac{\gamma}{16C_\chi},
  \label{eq:source-6-4}
\end{equation}
where \(C_\rho\) is a sufficiently large absolute constant. The gap budget determines the number of stages: let
\begin{equation}
  M=\left\lfloor c_M\frac{H\Delta}{\gamma^2}\right\rfloor
  \label{eq:source-6-5}
\end{equation}
for a sufficiently small constant \(c_M>0\).  Under \eqref{eq:source-4-1}, constants can be
chosen so that \(M\ge2\).

Draw independent hidden directions
\begin{equation}
  U_1,\ldots,U_M\stackrel{\rm iid}{\sim}\Normal(0,I_n/n)
  \label{eq:source-6-6}
\end{equation}
and define the good event
\begin{equation}
  \mathcal G=\left\{\max_{j\le M}\norm{U_j}_2\le2\right\}.
  \label{eq:source-6-7}
\end{equation}
Let \(c_1=0\) and recursively set
\begin{equation}
  c_{j+1}=c_j+\ell_{\rm clk}e_0+a_{\rm step}(0,U_j),
  \qquad 1\le j\le M.
  \label{eq:source-6-8}
\end{equation}
Equivalently,
\begin{equation}
  c_j=\left((j-1)\ell_{\rm clk},
  a_{\rm step}\sum_{i<j}U_i\right).
  \label{eq:source-6-9}
\end{equation}
The hidden rooms are
\begin{equation}
  B_j=B_2^d(c_j,\rho),
  \qquad 1\le j\le M+1.
  \label{eq:source-6-10}
\end{equation}
Their clock coordinates differ by \(4\rho\), hence the balls are pairwise
disjoint.  On \(\mathcal G\), the center norms are at most \(\sqrt{4^2+32^2}\,M\rho\). Since \(M\ge2\), all rooms are contained in the ball of radius \(40M\rho\). We choose
\begin{equation}
  R=C_{\rm ball}M\rho,\qquad C_{\rm ball}=40.
  \label{eq:source-6-11}
\end{equation}
Since \(M\asymp H\Delta/\eta^2\) and \(\rho\asymp\eta/H\), this radius is
\(\Theta(\Delta/\eta)\), as claimed in \eqref{eq:source-4-2}.

A sample handle \(h\) contains independent vectors
\begin{equation}
  Z_{h,1},\ldots,Z_{h,M}
  \stackrel{\rm iid}{\sim}\Normal(0,I_n/n).
  \label{eq:source-6-12}
\end{equation}
Choose the noise level to use the sample-Lipschitz budget, and define the rank threshold that will trigger a stage reveal in the proof:
\begin{equation}
  \sigma^2=c_\sigma\frac{L_0^2}{\gamma^2},
  \qquad
  q_0=\left\lfloor\frac{n\sigma^2}{16}\right\rfloor,
  \label{eq:source-6-13}
\end{equation}
where \(c_\sigma>0\) is a sufficiently small absolute constant.  The
constant \(c_0\) in \eqref{eq:source-4-1} is chosen later, after \(c_\sigma\), so that
\(\sigma^2\) and hence \(q_0/n\) are bounded below by an absolute constant.
The sample function combines the carrier, informative bumps, and terminal correction:
\begin{equation}
  F_U(x;h)
  =-\gamma s 
  +\lambda\rho\sum_{j=1}^M
  b_{U_j+\sigma Z_{h,j}}
  \left(\frac{x-c_j}{\rho}\right) 
  +\gamma\rho\,
  \tau_0\left(\frac{x-c_{M+1}}{\rho}\right).
  \label{eq:source-6-14}
\end{equation}
Because the supports are disjoint, at most one local term is active at any
point.  Taking expectation over a handle gives
\begin{equation}
  f_U(x)
  =-\gamma s 
  +\lambda\rho\sum_{j=1}^M
  b_{U_j}\left(\frac{x-c_j}{\rho}\right) 
  +\gamma\rho\,
  \tau_0\left(\frac{x-c_{M+1}}{\rho}\right).
  \label{eq:source-6-15}
\end{equation}

\subsection{Admissibility and the stationary set}\label{sec:verification}

The construction must satisfy the analytic bounds for every fixed admissible instance, while the information argument will use an unconditioned Gaussian prior. We therefore first isolate a high-probability event on which the entire family is admissible. The Gaussian estimates in \Cref{lem:gaussian-max}, proved in \Cref{sec:analytic-details}, control both the hidden directions and the sample-wise Lipschitz modulus.

We fix constants in the following order.  First choose the cutoff \(\chi\)
and hence \(C_\chi\), and fix \(C_{\rm ball}=40\). Next choose \(C_\rho\) large enough for smoothness,
then \(c_\sigma\) small enough for the sample-Lipschitz bound, then \(c_0\)
small enough that \(\eta\le c_0L_0\) implies \(\sigma^2\ge64\), and finally
choose \(c_M\) small and \(C_0\) in \eqref{eq:source-4-1} large.  With this hierarchy,
\eqref{eq:source-4-1} implies
\begin{equation}
  M\ge2,
  \qquad
  q_0\ge \frac{n\sigma^2}{32},
  \qquad
  n\ge C\log(2+Mq_0),
  \qquad
  \Prob(\mathcal G)\ge\frac78.
  \label{eq:source-7-3}
\end{equation}
Indeed,
\(M\asymp H\Delta/\eta^2\),
\(q_0\asymp dL_0^2/\eta^2\), and hence
\(Mq_0\lesssim dH\Delta L_0^2/\eta^4\); the logarithm of the extra factor
\(d\) is absorbed by the additive constant and a sufficiently large
constant in \eqref{eq:source-4-1}.  The same assumptions and the floor in \eqref{eq:source-6-5} give
\begin{equation}
  c_R\frac{\Delta}{\eta}
  \le R=C_{\rm ball}M\rho
  \le C_R\frac{\Delta}{\eta}
  \label{eq:source-7-4}
\end{equation}
for absolute constants. The choice \(C_{\rm ball}=40\) already ensures that every room lies in the ball on \(\mathcal G\).

\begin{lemma}[Smoothness, stochastic Lipschitz scale, gap, and stationary set]
\label{lem:class-verification}
For the constant hierarchy above, every realization \(U\in\mathcal G\)
satisfies
\begin{equation}
  \Lip(\nabla f_U;\R^d)\le H,
  \qquad
  \E_h\Lip(F_U(\cdot;h))^2\le L_0^2,
  \qquad
  f_U(0)-\min_{x\in X}f_U(x)\le\Delta.
  \label{eq:source-7-5}
\end{equation}
Moreover,
\begin{equation}
  x\notin B_{M+1}
  \quad\Longrightarrow\quad
  \norm{\nabla f_U(x)}_2>\eta,
  \label{eq:source-7-6}
\end{equation}
while \(\nabla f_U(c_{M+1})=0\).
\end{lemma}

\begin{proof}
The smoothness, sample-Lipschitz, and gap estimates are verified in \Cref{sec:class-estimates}. We retain the stationary-set argument here because it links admissibility to the information lower bound.

Outside the terminal room, the carrier contributes \(-\gamma e_0\).  In a
nonterminal room the local population bump has gradient norm at most
\begin{equation}
  C_\chi\lambda\norm{U_j}_2\le\frac{\gamma}{8}.
  \label{eq:source-7-12}
\end{equation}
Consequently,
\begin{equation}
  \norm{\nabla f_U(x)}_2
  \ge\frac{7\gamma}{8}
  =\frac72\eta>\eta
  \qquad (x\notin B_{M+1}).
  \label{eq:source-7-13}
\end{equation}
Finally, \(\nabla\tau_0(0)=e_0\), so at the terminal center the terminal
bump cancels the carrier exactly:
\(\nabla f_U(c_{M+1})=0\).
\end{proof}

Thus, on the valid-instance event \(\mathcal G\), every successful algorithm
must output a point in the terminal room.  The qualification by
\(\mathcal G\) is essential and will be retained explicitly in the final
probability argument.

\section{Proof of the Lower Bound}\label{sec:lower-proof}

The preceding construction reduces successful optimization to reaching the terminal room. We now show that each new room requires a substantial number of scalar observations, and that adaptive queries cannot bypass this cost. Combining these facts proves \Cref{thm:smooth-lower}; the final subsection derives \Cref{cor:goldstein-lower}. More precisely, under the unconditioned Gaussian prior we will prove, for every algorithm with budget \(N<Mq_0\),
\[
\Prob\bigl(\|\nabla f_U(x_{\rm out})\|_2\le\eta,\ U\in\mathcal G\bigr)
\le (N+1)e^{-cn}.
\]
Since \(\Prob(\mathcal G)\ge7/8\), this implies that some admissible instance has success probability below \(1/2\) for the budgets in the theorem. We keep the Gaussian prior unconditioned during the information argument and restrict to valid instances only at the final averaging step.

\subsection{Information required to resolve one room}\label{sec:information}

Fix a stage \(j\).  For a query \(x\in B_j\), write
\(y=(x-c_j)/\rho\).  Once the known carrier is subtracted, the informative
part of \eqref{eq:source-6-14} is
\begin{equation}
  \lambda\rho\,
  \inner{U_j+\sigma Z_{h,j}}{\chi(y)y_\perp}.
  \label{eq:source-8-1}
\end{equation}
Thus every nonzero query direction gives one exact scalar linear measurement
of \(V_{h,j}=U_j+\sigma Z_{h,j}\).  Granting the center \(c_j\) and allowing
arbitrary linear directions can only strengthen the oracle, so it is enough
to lower-bound this idealized adaptive projection experiment.

\begin{lemma}[Adaptive Gaussian projection posterior]
\label{lem:posterior}
Let
\(U\sim\Normal(0,I_n/n)\) and, independently,
\(Z_h\stackrel{\rm iid}{\sim}\Normal(0,I_n/n)\).  An adaptive experiment has
independent internal randomness and, at each time, may choose a handle and a
linear functional predictably from the past, observing that functional of
\(V_h=U+\sigma Z_h\).  For each handle, discard linearly redundant queries
and Gram--Schmidt the remaining directions; let \(P_h\) be the projector
onto their span.  Conditional on the complete transcript, including the
realized adaptive directions, the posterior law of \(U\) is Gaussian with
covariance
\begin{equation}
  \Sigma
  =\frac1n\left(I+\sigma^{-2}\sum_hP_h\right)^{-1}.
  \label{eq:source-8-2}
\end{equation}
If
\begin{equation}
  q:=\sum_h\rank(P_h),
  \label{eq:source-8-3}
\end{equation}
then each scalar evaluation increases \(q\) by at most one.
\end{lemma}

\begin{proof}
Fix one handle.  Suppose its previously retained innovation directions span
\(S\), and let a new nonredundant query have normalized innovation direction
\(a\in S^\perp\).  The component of the new observation in \(S\) is already
known from the earlier observations of the same vector \(V_h\), so the only
new datum is
\begin{equation}
  Y=\inner{U}{a}+\sigma\inner{Z_h}{a}.
  \label{eq:source-8-4}
\end{equation}
The direction \(a\) is measurable with respect to the past.  Conditional on
that past and on \(U\), Gaussian rotational invariance implies that
\(\inner{Z_h}{a}\sim\Normal(0,1/n)\) and is independent of all earlier
innovations.  The same conditional innovation statement holds when the next query uses a different handle, because its unobserved noise coordinates are independent of the other handles conditional on \(U\). The observations themselves are not asserted to be independent after marginalizing over \(U\).
Therefore the conditional likelihood at each step is the ordinary Gaussian
linear-regression likelihood with noise variance \(\sigma^2/n\).  The prior
precision is \(nI\), and each unit innovation direction contributes
\((n/\sigma^2)aa^\top\).  Multiplying the sequential conditional
likelihoods gives
\[
  \Sigma^{-1}
  =nI+\frac n{\sigma^2}\sum_hP_h,
\]
which is \eqref{eq:source-8-2}.  The fact that the directions are adaptive causes no extra
likelihood factor, because each direction is a predictable function of the
already conditioned-on transcript and independent internal randomness.
Finally, one scalar observation can create at most one new innovation
direction.
\end{proof}

\begin{lemma}[Posterior small-ball obstruction]
\label{lem:small-ball}
Let \(A=\sum_hP_h\).  If
\begin{equation}
  q=\tr(A)<q_0,
  \label{eq:source-8-5}
\end{equation}
then, for every transcript-measurable estimator \(\widehat u\),
\begin{equation}
  \Prob\!\left(\norm{U-\widehat u}_2\le\frac1{16}
  \,\middle|\,\mathcal T\right)
  \le e^{-cn}
  \label{eq:source-8-6}
\end{equation}
for an absolute constant \(c>0\).
\end{lemma}

\begin{proof}
Because \(q<q_0\le n\sigma^2/16\), fewer than \(n/4\) eigenvalues of \(A\) exceed \(\sigma^2/4\). By \eqref{eq:source-8-2}, there is a posterior covariance eigenspace of dimension \(k\ge3n/4\) on which all eigenvalues are at least \(a/n\), where \(a=4/5\). Conditional on the transcript, project \(U-\widehat u\) onto this space. Its Gaussian density, with arbitrary mean, is bounded by \((n/(2\pi a))^{k/2}\). Multiplying this bound by the volume of a radius-\(r\) ball, with \(r=1/16\), gives
\begin{equation}
\Prob\!\left(\|U-\widehat u\|_2\le r\mid\mathcal T\right)
\le \left(\frac{n}{2\pi a}\right)^{k/2}
       \frac{\pi^{k/2}r^k}{\Gamma(k/2+1)} 
\le \left(\frac{enr^2}{ak}\right)^{k/2}
\le e^{-cn}.
\label{eq:source-8-7}
\end{equation}
Here \(\Gamma(k/2+1)\ge(k/(2e))^{k/2}\), and \(n/k\le4/3\). The bound is uniform over the posterior mean and the estimator, so no centeredness assumption is needed.
\end{proof}

By \eqref{eq:source-6-13} and the constant hierarchy,
\begin{equation}
  q_0
  =\Theta(n\sigma^2)
  =\Theta\!\left(\frac{dL_0^2}{\eta^2}\right).
  \label{eq:source-8-8}
\end{equation}
This is the scalar-information cost of resolving one room, even with
arbitrary adaptive reuse of sample handles.

\subsection{Sequential traversal and the smooth lower bound}\label{sec:traversal}

\Cref{lem:small-ball} controls one unresolved direction, but the lower bound requires that this cost be paid at every stage. A direct argument conditioned on missing future regions would change the posterior. We instead construct a shadow transcript in which unresolved regions leave no trace, and transfer its traversal bound to the true oracle at the first possible disagreement.

\subsubsection{Augmented and shadow experiments}

Fix an arbitrary randomized algorithm \(\mathcal A\) with a deterministic
budget of \(N\) scalar evaluations; include its internal random seed in the
probability space. We regard its query and output rules as measurable maps defined on every finite transcript, including transcripts outside the promised instance class. On every run, query points and the output lie in \(X\), and the algorithm stops after at most \(N\) evaluations. These rules can be extended arbitrarily off the promised class, with a forced stop at \(N\), without changing their behavior on valid instances. This lets us run the same strategy under the unconditioned Gaussian prior and the shadow experiment below. An early-stopping algorithm can be padded with ignored queries at \(x_0\), retaining its original output, so the argument may use exactly \(N\) query slots.

We may strengthen the oracle as follows. It keeps a
current stage \(J\), initially \(J=1\).  For the current room, it tracks the
projection rank \(q_J\) from \cref{lem:posterior}.  As soon as
\(q_J=q_0\), it reveals \(U_J\) to the algorithm for free, constructs
\(c_{J+1}\), and increments the stage.  Giving these reveals can only help an
algorithm, so a lower bound in this augmented model implies the desired
lower bound.

Couple two augmented experiments using the same hidden vectors, handle
noises, and algorithmic randomness.

\begin{itemize}
\item The \emph{true experiment} answers with the full sample function
      \eqref{eq:source-6-14}, containing every room from the outset.
\item The \emph{shadow experiment}, while the current stage is \(J=j\),
      retains all bumps with index at most \(j\) and suppresses every bump in
      rooms \(j+1,\ldots,M+1\).  At an unlock it activates the next bump and
      uses the same pre-sampled handle noises as the true experiment.
\end{itemize}

The shadow experiment need not itself be a stationary sample-function oracle;
it is only a coupling device.  Until a shadow query enters a geometrically
unresolved future room, every scalar answer and every scheduled reveal in
the two experiments is identical.

Let \(\mathcal F_t\) be the shadow filtration just before the \(t\)-th query,
including the internal random seed, all previous query points and answers,
and all directions already revealed at unlocks.  Future-room membership is
not included in \(\mathcal F_t\): suppressed rooms always return only the
carrier in the shadow experiment.

\paragraph{The shadow filtration invariant.}
At a time with stage \(J=j\), the transcript uses only \(U_1,\ldots,U_j\), noise coordinates from rooms with index at most \(j\), and the algorithm's seed. Directions with larger indices and their room noises remain independent of this transcript. This follows inductively: an answer in a revealed room uses its revealed center and direction; an answer in the current room is a predictable projection as in \eqref{eq:source-8-1}; all other answers are the known carrier. Neither suppressed-room membership nor a condition that earlier queries missed those rooms is observed. At an unlock the revealed \(U_j\) may correlate with past answers, but \(U_{j+1}\) and its room noises were independent of all variables used so far. They therefore start the next stage with their original Gaussian laws. During that stage, only its projection likelihoods update its hidden direction. Later queries to earlier rooms use noise vectors independent of the new direction and do not add a likelihood factor for it. These statements hold on each event \(J=j\), so they also apply when the stage index and unlock times are random.

\subsubsection{One-step future-hit bound}

Suppose the current shadow stage is \(J=j\) and \(q_j<q_0\).  The current
center \(c_j\) is \(\mathcal F_t\)-measurable.  Conditional on
\(\mathcal F_t\), the posterior of \(U_j\) has the form in
\cref{lem:posterior}; replies obtained in past rooms depend only on already
revealed directions and independent noises, while replies outside the
current and past rooms are the deterministic carrier.

For \(k>j\),
\begin{equation}
  c_k=c_j+(k-j)\ell_{\rm clk}e_0
      +a_{\rm step}\left(0,\sum_{r=j}^{k-1}U_r\right).
  \label{eq:source-9-1}
\end{equation}
The clock-coordinate intervals of the balls \(B_k\) are disjoint because
\(\ell_{\rm clk}=4\rho\).  Hence a fixed query point can be compatible with
at most one future index \(k\).  For that index, a hit \(x\in B_k\) implies
\begin{equation}
  \norm{
  \frac{x_\perp-c_{j,\perp}}{a_{\rm step}}
  -\sum_{r=j+1}^{k-1}U_r-U_j
  }_2
  \le\frac{\rho}{a_{\rm step}}
  =\frac1{16}.
  \label{eq:source-9-2}
\end{equation}
Condition additionally on the independent future directions
\(U_{j+1},\ldots,U_{k-1}\).  The first two terms in \eqref{eq:source-9-2}, together with the
conditioned future sum, form an estimator of \(U_j\) that is measurable with
respect to the enlarged conditioning sigma-field.  Since those future
directions are independent of \((U_j,\mathcal F_t)\),
\cref{lem:small-ball} applies uniformly and then remains valid after
integrating them out.

\begin{lemma}[Adaptive future-hit bound]
\label{lem:early-hit}
At any shadow time with current stage \(j\) and \(q_j<q_0\),
\begin{equation}
  \Prob\!\left(
  x_t\in\bigcup_{k>j}B_k
  \,\middle|\,\mathcal F_t\right)
  \le e^{-cn}.
  \label{eq:source-9-3}
\end{equation}
The same bound holds for the final output, conditional on the shadow
filtration at termination, whenever the terminal room is still unresolved.
\end{lemma}

\begin{proof}
The query \(x_t\) is \(\mathcal F_t\)-measurable.  Its clock coordinate
selects at most one future index.  If there is none, the conditional
probability is zero.  Otherwise \eqref{eq:source-9-2}, conditioning on the intervening future
directions, and \cref{lem:small-ball} give the bound.  The output is another
filtration-measurable point and is treated identically.
\end{proof}

Let \(E\) be the event, defined on the shadow trajectory, that some query or
the final output belongs to a room whose index is larger than the current
shadow stage at that time.  Importantly, we do not condition on previous
misses.  Applying \cref{lem:early-hit} directly to every shadow proposal and
using the tower property gives the adaptive union bound
\begin{equation}
  \Prob(E)
  \le\sum_{t=1}^N
      \E\!\left[
      \Prob(E_t\mid\mathcal F_t)
      \right]
      +e^{-cn}
  \le(N+1)e^{-cn}.
  \label{eq:source-9-4}
\end{equation}
A suppressed-room miss contributes no membership information, so the posterior estimate remains applicable at every shadow proposal.

\subsubsection{Coupling and conditioning on the valid-instance event}

On \(E^c\), the true and shadow augmented transcripts are identical by
induction.  Indeed, before an unlock they differ only inside unresolved
future rooms, and at an unlock they reveal the same \(U_j\), activate the
same next bump, and continue with the same retained handle noises.  Each unlock requires \(q_0\) rank increments summed over the retained handles: within a handle, only a direction outside its previously observed span is counted, whereas different handles contribute separately even along the same direction. One scalar evaluation increases this sum by at most one.  Therefore, if
\begin{equation}
  N<Mq_0,
  \label{eq:source-9-5}
\end{equation}
the shadow experiment cannot unlock all \(M\) hidden directions.  Its final
output cannot lie in the terminal room on \(E^c\); by coupling, neither can
the true output.

Let
\[
  S=\{\norm{\nabla f_U(x_{\rm out})}_2\le\eta\}.
\]
The implication \(S\subseteq\{x_{\rm out}\in B_{M+1}\}\) is known only on
\(\mathcal G\), not for every Gaussian realization.  The correct bound is
therefore
\begin{equation}
  \Prob(S\cap\mathcal G)
  \le \Prob(E)
  \le (N+1)e^{-cn}.
  \label{eq:source-9-6}
\end{equation}

By \eqref{eq:source-7-3}, \(n\ge C\log(2+Mq_0)\) and
\(\Prob(\mathcal G)\ge7/8\).  Hence, for a sufficiently small absolute
\(c_N>0\), every budget \(N\le c_NMq_0\) satisfies
\begin{equation}
  (N+1)e^{-cn}\le\frac14.
  \label{eq:source-9-7}
\end{equation}
For the fixed algorithm \(\mathcal A\), let
\(p_{\mathcal A}(U)\) be its success probability on the fixed instance
indexed by \(U\), including handle noise and its own randomization.  From
\eqref{eq:source-9-6}--\eqref{eq:source-9-7},
\begin{equation}
  \E\!\left[p_{\mathcal A}(U)\,\middle|\,\mathcal G\right]
  =\frac{\Prob(S\cap\mathcal G)}{\Prob(\mathcal G)}
  \le\frac{1/4}{7/8}
  =\frac27<\frac12.
  \label{eq:source-9-8}
\end{equation}
Consequently, there exists a fixed realization \(U\in\mathcal G\) for which
\(p_{\mathcal A}(U)<1/2\).  By
\cref{lem:class-verification}, this realization is a valid instance.

Finally,
\begin{equation}
  M=\Theta\!\left(\frac{H\Delta}{\eta^2}\right),
  \qquad
  q_0=\Theta\!\left(\frac{dL_0^2}{\eta^2}\right),
  \label{eq:source-9-9}
\end{equation}
so
\begin{equation}
  Mq_0
  =\Theta\!\left(\frac{dL_0^2H\Delta}{\eta^4}\right).
  \label{eq:source-9-10}
\end{equation}
Because \(\mathcal A\) was arbitrary, this proves
\cref{thm:smooth-lower}.

\subsection{From smooth stationarity to Goldstein stationarity}\label{sec:reduction}

We prove \cref{cor:goldstein-lower}.  The following elementary implication
is the only reduction needed.

\begin{lemma}[Goldstein residual controls the gradient of a smooth function]
\label{lem:goldstein-gradient}
If \(f\) has \(H\)-Lipschitz gradient, then for every \(x\),
\begin{equation}
  \norm{\nabla f(x)}_2\le G_\delta(f,x)+H\delta.
  \label{eq:source-10-1}
\end{equation}
\end{lemma}

\begin{proof}
For any \(g\in\partial_\delta f(x)\), there are points
\(y_i\in B_2^d(x,\delta)\) and convex weights \(\lambda_i\) such that
\(g=\sum_i\lambda_i\nabla f(y_i)\), up to an arbitrarily small closure
error.  Hence
\begin{equation}
  \norm{\nabla f(x)}_2
  \le\norm{g}_2+
  \sum_i\lambda_i\norm{\nabla f(x)-\nabla f(y_i)}_2 
  \le\norm{g}_2+H\delta.
  \label{eq:source-10-2}
\end{equation}
Taking the infimum over \(g\in\partial_\delta f(x)\) proves the claim.
\end{proof}

Apply \cref{thm:smooth-lower} with
\begin{equation}
  H=\frac{\eps}{\delta},
  \qquad
  \eta=2\eps.
  \label{eq:source-10-3}
\end{equation}
If \(G_\delta(f,x)\le\eps\), then
\cref{lem:goldstein-gradient} gives
\(\norm{\nabla f(x)}_2\le2\eps=\eta\).  The assumptions in \eqref{eq:source-4-1} become,
up to absolute constants, precisely \eqref{eq:source-4-5}, while \eqref{eq:source-4-4} becomes
\begin{equation}
  \Omega\!\left(
  \frac{dL_0^2(\eps/\delta)\Delta}{(2\eps)^4}
  \right)
  =
  \Omega\!\left(\frac{dL_0^2\Delta}{\delta\eps^3}\right).
  \label{eq:source-10-4}
\end{equation}
The radius in \eqref{eq:source-4-2} becomes \(R=\Theta(\Delta/\eps)\). Every hard instance has \(\nabla f(x_*)=0\) at its terminal center in the ball, so it also has \(0\in\partial^C f(x_*)\). Thus the family belongs to the promised subclass defining \(N_G^{\mathrm{stat}}\), completing the proof.

\section{Discussion}\label{sec:discussion}

The construction separates sequential progress from the information required to make that progress. Under \(H=\eps/\delta\), the gap budget supports \(\Theta(\Delta/(\delta\eps))\) hidden stages. Each stage requires \(\Theta(dL_0^2/\eps^2)\) scalar observations even with adaptive sample reuse. The logarithmic dimension threshold controls Gaussian norms and the union bound over premature hits; it does not contribute a logarithmic factor to the query lower bound.

The existence promise separates this information obstruction from the geometric impossibility in \Cref{prop:radius-obstruction}. Every hard instance contains a stationary point at the terminal center. Nevertheless, the linear carrier makes its global infimum equal to \(-\infty\), and increasing the query radius increases the local gap. Consequently, the construction cannot be transferred to an unrestricted global-gap model by simply enlarging its ball. The auxiliary localization in \Cref{sec:upper} also requires a larger radius, so it does not settle the complexity of the promised class at the constructed radius.

Two questions remain: whether the same product lower bound holds in the globally bounded-gap model against arbitrary randomized value-query algorithms, and whether a comparable finite upper bound can be proved for an appropriate promised class at the lower-bound radius. Either requires further work. The present result establishes a joint lower bound on solvable bounded-domain instances; it does not settle either question.

\section{Conclusion}\label{sec:conclusion}

We established a joint zeroth-order lower bound for finding Goldstein stationary points on a Euclidean query ball, even when each instance has a stationary point in that ball. Spatially separated regions enforce sequential learning, while an adaptive Gaussian posterior calculation controls arbitrary reuse of sampled functions. The result makes the dimension and accuracy costs arise jointly on the same family, despite adaptive sample reuse. Its scope remains the local-gap, bounded-query model; extending the lower bound to unrestricted queries under a global gap is an open direction.

\clearpage
\appendix
\section{Localization of the Upper Bound}\label{sec:upper}

This appendix records a localization of the Kornowski--Shamir method~\citep{KornowskiShamir2024}. It retains the established evaluation rate while replacing the global gap by a gap over a sufficiently large query ball. It is an auxiliary consequence, and its radius condition differs from that of the lower-bound construction.

\begin{proposition}[Localization of the two-point upper bound]
\label{prop:localized-upper}
Let \(d\ge1\), \(L_0,\delta>0\), \(\Delta\ge0\), and
\(0<\eps\le L_0\).  There are absolute constants \(C_0,C_1>0\) such that
the following holds.  Let \(X=B_2^d(x_0,R)\), and suppose
\begin{equation}
  R\ge C_0\left(\frac{\Delta}{\eps}+\delta\right).
  \label{eq:source-4-7}
\end{equation}
Then a fresh common-sample two-point algorithm, and hence an algorithm in the
reusable-handle model, uses at most
\begin{equation}
  C_1\left(1+\frac{dL_0^2\Delta}{\delta\eps^3}\right)
  \label{eq:source-4-8}
\end{equation}
scalar evaluations and returns \(x_{\rm out}\in X\) with
\begin{equation}
  \Prob\bigl(G_\delta(f,x_{\rm out})\le\eps\bigr)\ge\frac12.
  \label{eq:source-4-9}
\end{equation}
Only the bounded-domain gap
\(f(x_0)-\min_Xf\le\Delta\) is required.
\end{proposition}

\begin{proof}
The key device is a projection extension, which supplies a global gap for an auxiliary objective while keeping the algorithm's queries inside the original ball.

Let \(\Pi_X\) be Euclidean projection onto the closed ball \(X\), and define
\begin{equation}
  \widetilde F(x;\xi)=F(\Pi_Xx;\xi),
  \qquad
  \widetilde f(x)=\E_\xi\widetilde F(x;\xi)=f(\Pi_Xx).
  \label{eq:source-11-1}
\end{equation}
Projection is nonexpansive, so
\begin{equation}
  \Lip(\widetilde F(\cdot;\xi))\le L(\xi),
  \qquad
  \E\Lip(\widetilde F(\cdot;\xi))^2\le L_0^2.
  \label{eq:source-11-2}
\end{equation}
Moreover,
\begin{equation}
  \widetilde f(x_0)-\inf_{x\in\R^d}\widetilde f(x)
  =f(x_0)-\min_{x\in X}f(x)
  \le\Delta.
  \label{eq:source-11-3}
\end{equation}
Thus the global Kornowski--Shamir theorem~\citep[Theorem 5]{KornowskiShamir2024} applies to the projected stochastic
objective.  To remove its normalization \(\delta,\eps\in(0,1)\), apply it to
\begin{equation}
  \widehat F(y;\xi)
  =\frac{\widetilde F(x_0+2\delta y;\xi)}{2\delta L_0}.
  \label{eq:source-11-4}
\end{equation}
The normalized Goldstein radius is \(1/2\), the normalized target corresponding
to expectation \(\eps/2\) is \(\eps/(2L_0)\le1/2\), and scaling back preserves
the quantity \(dL_0^2\Delta/(\delta\eps^3)\) up to an absolute constant.

Before invoking the algorithm, we handle the small-gap regime directly.  If \(B_2^d(x,\delta)\subset X\), then
\begin{equation}
  \delta G_\delta(f,x)
  \le f(x)-\min_{y\in X}f(y).
  \label{eq:source-11-5}
\end{equation}
To see this, let \(g_*\) be the minimum-norm point of
\(\partial_\delta f(x)\), and set \(u=-g_*/\norm{g_*}\) when \(g_*\ne0\).
The projection characterization gives
\(\inner{g}{u}\le-\norm{g_*}\) for every
\(g\in\partial_\delta f(x)\).  The restriction \(\phi(t)=f(x+tu)\) is Lipschitz and hence absolutely continuous on \([0,\delta]\). Almost everywhere its derivative is bounded above by the Clarke directional derivative \(f^\circ(x+tu;u)=\max_{v\in\partial^C f(x+tu)}\langle v,u\rangle\le-\|g_*\|\). Integrating gives \(f(x+\delta u)\le f(x)-\delta\|g_*\|\), proving \eqref{eq:source-11-5}; the case \(g_*=0\) is immediate.

Put
\begin{equation}
  Q:=\frac{dL_0^2\Delta}{\delta\eps^3}.
  \label{eq:source-11-6}
\end{equation}
If \(\Delta\le\delta\eps\), then \eqref{eq:source-4-7} ensures \(B_2^d(x_0,\delta)\subset X\), and \eqref{eq:source-11-5} gives \(G_\delta(f,x_0)\le\eps\). Return \(x_0\) without making an oracle call. This includes \(\Delta=0\).

It remains to consider \(\Delta>\delta\eps\). In particular \(Q>dL_0^2/\eps^2\ge1\). Run the Kornowski--Shamir method with
target expectation \(\eps/2\).  It chooses
\begin{equation}
  \rho_s=\min\left\{\frac\delta2,\frac\Delta{L_0}\right\},
  \qquad
  \nu=\delta-\rho_s
  =\max\left\{\frac\delta2,
  \delta-\frac\Delta{L_0}\right\},
  \label{eq:source-11-7}
\end{equation}
Let \(h\) be the uniform \(\rho_s\)-smoothing of \(\widetilde f\). Its initial global gap is at most \(\Delta_h=\Delta+\rho_sL_0\). The common-sample estimator in Kornowski and Shamir~\citep[Lemma 7]{KornowskiShamir2024} is unbiased for \(\nabla h\) and has second moment at most \(V=C_gdL_0^2\), for an absolute \(C_g\ge1\). Choose
\begin{equation}
  T=\lceil C_T C_g Q\rceil\le(C_TC_g+1)Q,
  \label{eq:source-11-8}
\end{equation}
and the clipped increment radius and update step size
\begin{equation}
  D=
  \left(
  \frac{\Delta_h\sqrt\nu}
  {\sqrt V\,T}
  \right)^{2/3},\qquad \alpha=\frac{\Delta_h}{VT},
  \qquad
  \Delta_h:=\Delta+\rho_sL_0\le2\Delta.
  \label{eq:source-11-9}
\end{equation}
Here \(C_T\) is a sufficiently large absolute constant. To check the integer parameters of the external algorithm, write its block length as \(m=\lfloor\nu/D\rfloor\) and its block count as \(K=\lfloor T/m\rfloor\). Since \(\delta/2\le\nu\le\delta\), \(\Delta\le\Delta_h\le2\Delta\), and \(T\ge C_TV\Delta/(\delta\eps^3)\),
\[
\frac D\nu=\left(\frac{\Delta_h}{\sqrt V T\nu}\right)^{2/3}
\le\left(\frac4{C_T}\right)^{2/3}\frac{\eps^2}{V}
\le\frac12.
\]
Also,
\[
TD=\left(\frac{\Delta_h^2\nu T}{V}\right)^{1/3}
\ge\left(\frac{C_T}2\right)^{1/3}\frac\Delta\eps
\ge2\nu
\]
after increasing \(C_T\). Thus \(2\le m\le T/2\), so \(K\ge2\) and the output selection is defined. For clarity, we verify the expectation guarantee with these rounded parameters rather than relying on an asymptotic choice of iteration count.

Initialize \(\Delta_1=0\). At iteration \(t\), set \(x_t=x_{t-1}+\Delta_t\), draw \(s_t\) uniformly from \([0,1]\), and set \(z_t=x_{t-1}+s_t\Delta_t\). Independently draw a fresh sample \(\xi_t\) and a uniform unit vector \(v_t\), and compute
\[
g_t=\frac{d}{2\rho_s}\bigl[\widetilde F(z_t+\rho_sv_t;\xi_t)
 -\widetilde F(z_t-\rho_sv_t;\xi_t)\bigr]v_t.
\]
The two scalar evaluations use the same sample. Update
\[
\Delta_{t+1}=\proj_{B(0,D)}(\Delta_t-\alpha g_t).
\]
For each of the \(K\) complete blocks of length \(m\), form the average of its centers \(z_t\), and return a uniformly selected block average. Let \(T'=Km\). Since \(m\le T/2\), we have \(T'\ge T/2\).

Here is the block estimate underlying the external method. For a block \(I\), projected gradient descent gives, pathwise for every \(u\in B(0,D)\),
\[
\sum_{t\in I}\langle g_t,\Delta_t-u\rangle
\le\frac{2D^2}{\alpha}+\frac\alpha2\sum_{t\in I}\|g_t\|^2.
\]
Choose \(u\) opposite to \(\sum_{t\in I}\nabla h(z_t)\), with norm \(D\) (or zero if the sum vanishes). The differences \(g_t-\nabla h(z_t)\) are martingale differences with second moments at most \(V\). Cauchy--Schwarz therefore bounds the expected norm of their block sum by \(\sqrt{mV}\), even though the chosen comparator depends on the whole block. Taking expectations in the preceding pathwise inequality yields
\[
D\,\E\left\|\sum_{t\in I}\nabla h(z_t)\right\|
\le-\E\sum_{t\in I}\langle g_t,\Delta_t\rangle
 +\frac{2D^2}{\alpha}+\frac{\alpha mV}{2}+D\sqrt{mV}.
\]
Unbiasedness and the uniform segment point imply
\(\E\langle g_t,\Delta_t\rangle=\E[h(x_t)-h(x_{t-1})]\).
Sum over the complete blocks and use the global gap of \(h\). Every center in a block is within distance \(mD\le\nu\) of its block average, so its average gradient belongs to the corresponding Goldstein subdifferential. Consequently,
\[
\begin{aligned}
\E G_\nu(h,x_{\rm out})
&\le\frac{\Delta_h}{T'D}+\frac{2D}{\alpha m}
   +\frac{\alpha V}{2D}+\sqrt{\frac Vm}\\
&\le 8\left(\frac{\Delta_h V}{T\nu}\right)^{1/3}
\le\frac\eps2.
\end{aligned}
\]
The second inequality uses \(T'\ge T/2\), \(m\ge\nu/(2D)\), and \eqref{eq:source-11-9}; the last holds for \(C_T\ge2^{14}\), since \(\Delta_h\le2\Delta\) and \(\nu\ge\delta/2\). The smoothing inclusion \(\partial_\nu h(x)\subseteq\partial_{\nu+\rho_s}\widetilde f(x)\) from Kornowski and Shamir~\citep[Lemma 4]{KornowskiShamir2024} then gives \(\E G_\delta(\widetilde f,x_{\rm out})\le\eps/2\).

It remains to transfer this run to the original ball. Every increment has norm at most \(D\), so every center is within \(TD\) of \(x_0\). The output is an average of such centers and obeys the same bound. Every function query lies within \(TD+\rho_s\).

Substituting \(T=O(Q)\), \(\Delta_h\le2\Delta\), and
\(\nu\le\delta\) into \eqref{eq:source-11-9} gives
\begin{equation}
  TD
  \le C
  Q^{1/3}
  \frac{\Delta^{2/3}\delta^{1/3}}
       {d^{1/3}L_0^{2/3}}
  \le C\frac\Delta\eps.
  \label{eq:source-11-10}
\end{equation}

Choose the constant in \eqref{eq:source-4-7} so that
\begin{equation}
  R\ge TD+2\delta.
  \label{eq:source-11-11}
\end{equation}
Since \(\rho_s\le\delta/2\), the deterministic trajectory bound then
implies that every query made by the algorithm on \(\widetilde F\) actually
lies in \(X\).  At such points
\(\widetilde F=F\), so running the same code with the original oracle produces
exactly the same transcript and output.  In addition, the output is at least
\(2\delta\) away from the boundary of \(X\).  Therefore \(f\) and
\(\widetilde f\) agree on an open neighborhood of every point in
\(B_2^d(x_{\rm out},\delta)\), which implies
\begin{equation}
  \partial^C\widetilde f(y)=\partial^Cf(y)
  \quad\text{for all }\norm{y-x_{\rm out}}\le\delta,
  \qquad
  G_\delta(\widetilde f,x_{\rm out})
  =G_\delta(f,x_{\rm out}).
  \label{eq:source-11-12}
\end{equation}
The expectation bound established above therefore gives
\begin{equation}
  \E G_\delta(f,x_{\rm out})
  =\E G_\delta(\widetilde f,x_{\rm out})
  \le\frac\eps2.
  \label{eq:source-11-13}
\end{equation}
It uses \(2T\) scalar evaluations.  Markov's inequality yields
\begin{equation}
  \Prob\bigl(G_\delta(f,x_{\rm out})>\eps\bigr)
  \le\frac12,
  \label{eq:source-11-14}
\end{equation}
which proves \cref{prop:localized-upper}.
\end{proof}

\section{Analytic Details for the Hard Family}\label{sec:analytic-details}

This appendix supplies the dimension-independent cutoff estimates and Gaussian bounds used in \Cref{sec:verification}, followed by the analytic part of the proof of \Cref{lem:class-verification}.

\subsection{Cutoff derivative bounds}

Write \(r=\|y\|^2\) and \(\chi(y)=\psi(r)\), with the fixed profile from \Cref{sec:construction}. On its support, \(\|y\|\le1\), and
\[
\nabla\chi(y)=2\psi'(r)y,\qquad
\nabla^2\chi(y)=2\psi'(r)I+4\psi''(r)yy^\top.
\]
Their norms are bounded by constants depending only on the fixed profile. For \(w=(0,v)\), the product rule gives
\[
\begin{aligned}
\nabla b_v(y)&=\chi(y)w+\langle w,y\rangle\nabla\chi(y),\\
\nabla^2 b_v(y)&=\langle w,y\rangle\nabla^2\chi(y)
 +w\nabla\chi(y)^\top+\nabla\chi(y)w^\top.
\end{aligned}
\]
Since \(|\langle w,y\rangle|\le\|v\|\) on the support and each rank-one term has operator norm at most the product of its vector norms, these formulas prove the bounds for \(b_v\) in \eqref{eq:source-6-3}. Taking \(w=e_0\) proves the bounds for \(\tau_0\). Outside the support all derivatives vanish. No bound depends on the dimension.

\subsection{Gaussian norm estimates}

\begin{lemma}[Gaussian norm and maximum bounds]
\label{lem:gaussian-max}
There are absolute constants \(c,C>0\) such that, for independent
\(G_1,\ldots,G_M\sim\Normal(0,I_n/n)\),
\begin{equation}
  \Prob\!\left(\max_{j\le M}\norm{G_j}_2>2\right)
  \le Me^{-cn},
  \label{eq:source-7-1}
\end{equation}
and, whenever \(n\ge C\log(2+M)\),
\begin{equation}
  \E\max_{j\le M}\norm{G_j}_2^2\le C.
  \label{eq:source-7-2}
\end{equation}
\end{lemma}

\begin{proof}
The chi-square tail estimate
\(\Prob(\norm{G_j}_2^2\ge1+t)\le
\exp[-cn\min\{t,t^2\}]\) gives \eqref{eq:source-7-1} by a union bound.  Integrating the
same union bound can be made explicit. For \(Y=\max_j\|G_j\|_2^2\) and \(b=2+2\log(2+M)/(cn)\), the upper tail for \(u\ge b\) gives
\[
\E Y\le b+\int_b^\infty M e^{-cn(u-1)}\,du
=b+\frac{M e^{-cn(b-1)}}{cn}\le C
\]
when \(n\ge C\log(2+M)\). This proves \eqref{eq:source-7-2}.
\end{proof}

\subsection{Smoothness, stochastic Lipschitz scale, and local gap}\label{sec:class-estimates}

We verify the analytic claims in \eqref{eq:source-7-5} for every fixed \(U\in\mathcal G\), using the constant hierarchy from \Cref{sec:verification}.

A nonterminal population bump has Hessian norm at most
\[
  \frac{C_\chi\lambda\norm{U_j}_2}{\rho},
\]
and the terminal bump has Hessian norm at most
\(C_\chi\gamma/\rho\).  The supports are pairwise disjoint, and on
\(\mathcal G\), \(\max_j\norm{U_j}\le2\).  Therefore, at every point,
\begin{equation}
  \norm{\nabla^2f_U(x)}_{\op}
  \le C\frac{\gamma}{\rho}
  =\frac{C}{C_\rho}H\le H
  \label{eq:source-7-7}
\end{equation}
when \(C_\rho\) is sufficiently large.  This is a global bound, not merely
a bound on \(X\).

Because the supports are disjoint, a sample-wise Lipschitz modulus satisfies
\begin{equation}
  L(h)
  \le C\gamma
  +C\lambda\max_{j\le M}\norm{U_j+\sigma Z_{h,j}}_2.
  \label{eq:source-7-8}
\end{equation}
For fixed \(U\in\mathcal G\),
\begin{equation}
  \E_h\max_{j\le M}\norm{U_j+\sigma Z_{h,j}}_2^2
  \le 2\max_j\norm{U_j}_2^2
       +2\sigma^2\E_h\max_j\norm{Z_{h,j}}_2^2 
  \le C(1+\sigma^2)
  \label{eq:source-7-9}
\end{equation}
by \cref{lem:gaussian-max}.  Hence
\begin{equation}
  \E_hL(h)^2
  \le C\gamma^2+C\lambda^2(1+\sigma^2)
  \le C\eta^2+C c_\sigma L_0^2
  \le L_0^2,
  \label{eq:source-7-10}
\end{equation}
by first choosing \(c_\sigma\) and then \(c_0\) sufficiently small.

At the origin all bump values vanish, so \(f_U(0)=0\).  On \(X\), the
carrier is bounded below by \(-\gamma R\), and at most one local term is
active.  Using \eqref{eq:source-6-3}, \(\max_j\norm{U_j}\le2\), and
\(R=C_{\rm ball}M\rho\),
\begin{equation}
  f_U(0)-\min_X f_U
  \le \gamma R+2\lambda\rho+\gamma\rho 
  \le C M\gamma\rho+C\gamma\rho.
  \label{eq:source-7-11}
\end{equation}
Now
\(M\gamma\rho\le c_M C_\rho\Delta\), while
\(\gamma\rho=C_\rho\gamma^2/H\le C\Delta/C_0\) by the second condition in
\eqref{eq:source-4-1}.  Thus a sufficiently small \(c_M\) and sufficiently large \(C_0\)
give the gap bound.

\end{document}